\documentclass[10 pt,reqno]{amsart}
 \usepackage{amsmath,amssymb}
\usepackage{latexsym}
\numberwithin{equation}{section}
\newtheorem{theorem}{Theorem}[section]
\newtheorem{corollary}[theorem]{Corollary}
\newtheorem{definition}[theorem]{Definition}
\newtheorem{lemma}[theorem]{Lemma}
\newtheorem{proposition}[theorem]{Proposition}
\newtheorem{remark}[theorem]{Remark}
\title{On long time dynamics of perturbed KdV equations}
\author{HUANG Guan\\C.M.L.S, \'Ecole polytechnique}
\begin{document}
\maketitle
\begin{abstract}Consider a perturbed KdV equation: 
\[u_t+u_{xxx}-6uu_x=\epsilon f(u(\cdot)),\quad x\in\mathbb{T}=\mathbb{R}/\mathbb{Z},\;\int_{\mathbb{T}}u(x,t)dx=0,\]
where the nonlinear perturbation defines analytic operators $u(\cdot)\mapsto f(u(\cdot))$ in sufficiently smooth Sobolev spaces.   Assume that the equation has an $\epsilon$-quasi-invariant measure $\mu$ and satisfies some additional mild assumptions. Let $u^{\epsilon}(t)$ be a solution. Then on time intervals of order $\epsilon^{-1}$, as $\epsilon\to0$, its actions $I(u^{\epsilon}(t,\cdot))$ can be approximated by solutions of a certain well-posed averaged equation, provided  that the initial datum is $\mu$-typical.
\end{abstract}
\bibliographystyle{plain}
\setcounter{section}{-1}
\section{Introduction}
The KdV equation on the circle, perturbed by  smoothing perturbations, was studied in \cite{hg2013}. There an averaging theorem that describes the long-time behavior for solutions of the perturbed KdV equation was  proved.
In this work, we suggest an abstract theorem which applies to a large class of $\epsilon$-perturbed KdV equations which  have $\epsilon$-quasi-invariant measures; the latter notion is explained in the main text. We show that the systems considered in \cite{hg2013}, satisfy this condition, and believe that it may be verified for many other perturbations of  KdV. More exactly, we consider a perturbed KdV equation with zero mean-value periodic boundary condition:
 \begin{equation}
 u_t+u_{xxx}-6uu_x=\epsilon f(u(\cdot)),\;  x\in \mathbb{T}=\mathbb{R}/\mathbb{Z}, \; \int_{\mathbb{T}}u(x,t)dx=0,
 \label{pkdv1}
 \end{equation}
where $\epsilon f$ is a nonlinear perturbation, specified below. For any $p\in\mathbb{R}$ we introduce the Sobolev space of real valued function on $\mathbb{T}$ with zero mean-value:
\[H^p=\Big\{u\in L^2(\mathbb{T},\mathbb{R}):\;||u||_p<+\infty,\;\int_{T}udx=0\Big\},\;||u||_p^2=\sum_{k\in\mathbb{N}}(2\pi k)^{2p}(|\hat{u}_k|^2+\hat{u}_{-k}^2).
\]
Here $\hat{u}_k$ and  $\hat{u}_{-k}$, $k\in\mathbb{N}$ are the Fourier coefficients of $u$ with respect to the trigonometric base
\[e_k=\sqrt{2}\cos 2\pi k x,\; k>0 \quad \text{and}\quad e_k=\sqrt{2}\sin 2\pi kx,\;k<0,\]
i.e. $u=\sum_{k\in\mathbb{N}}\hat{u}_ke_k+\hat{u}_{-k}e_{-k}$.
 It is well known that KdV is integrable. It means that the function space $H^p$ admits analytic coordinates $v=(\mathbf{v}_1,\mathbf{v}_2,\dots)=\Psi(u(\cdot))$, where $\mathbf{v}_j=(v_j,v_{-j})^t\in\mathbb{R}^2$, such that the  quantities $I_j=\frac{1}{2}|\mathbf{v}_j|^2$ and $\varphi_j= \text{Arg}\;\mathbf{v}_j$, $j\geqslant1$,  are action-angle variables for KdV.  In the $(I,\varphi)$-varibles, KdV takes the integrable form
 \begin{equation}\dot{I}=0,\quad \dot{\varphi}=W(I),
 \label{iw-form}
 \end{equation}
 where $W(I)\in\mathbb{R}^{\infty}$ is the KdV frequency vector, see \cite{kjp2003,hk2013}.  For any $p\geqslant 0$ the integrating transformation $\Psi$ (the nonlinear Fourier transform) defines an analytic diffeomorphism \mbox{$\Psi: H^p\to h^p$}, where 
 \[h^p=\Big\{v=(\mathbf{v}_1,\mathbf{v}_2,\dots): |v|_p^2=\sum_{j=1}^{\infty}(2\pi j)^{2p+1}|\mathbf{v}_j|^2<\infty, \mathbf{v}_j=(v_j,v_{-j})^t\in\mathbb{R}^2\Big\}.\]
 We introduce the weighted $l^1$-space $h^p_I$,
 \[h^p_I=\Big\{ I=(I_1,I_2,\dots)\in\mathbb{R}^{\infty}: |I|^{\sim}_p=2\sum_{j=1}^{+\infty}(2\pi j)^{2p+1}|I_j|<\infty\Big\},
 \]
 and  the mapping  $\pi_I$:
 \begin{equation}\pi_I: h^p\to h^p_I, \quad (\mathbf{v}_1,\dots)\mapsto(I_1,\dots), \;I_j=\frac{1}{2}\mathbf{v}_j^t\mathbf{v}_j,\; j\in\mathbb{N}.
 \label{maps}
 \end{equation}
 Obviously, $\pi_I$ is continuous, $|\pi_I(v)|_p^{\sim}=|v|_p^2$ and its image $ h^p_{I+}=\pi_I(h^p)$ is the positive octant of $h^p_I$. 
 
 We wish to study the long-time behavior of solutions for equation (\ref{pkdv1}).  Accordingly,  we fix some \[\zeta_0\geqslant0,\quad  p\geqslant3,\quad T>0,\] and  assume
 \smallskip
 
 \noindent{\bf Assumption A}: {\it (i) For any  $u_0\in H^{p}$, there exists a unique solution $u(t)\in H^{p}$ of (\ref{pkdv1}) with $u(0)=u_0$. It satisfies
 \[||u||_{p}\leqslant C(T,p,||u_0||_{p}), \quad 0\leqslant t\leqslant T\epsilon^{-1}.\]
 
 (ii) There exists a $p^{\prime}=p^{\prime}(p)<p$  such that for $q\in[p^{\prime},p]$,    the perturbation term defines an analytic mapping 
 \[ H^q\to H^{q+\zeta_0},\quad  u(\cdot)\mapsto f(u(\cdot)).\]
 }
 \smallskip
 
 We are mainly concerned with the behavior of the actions $I(u(t))\in\mathbb{R}^{\infty}_+$ on time interval $ [0,T\epsilon^{-1}]$.  For this end,  we write  the perturbed KdV (\ref{pkdv1}), using  slow time $\tau=\epsilon t$ and  the  $v$-variables:
 \begin{equation}
  \frac{dv}{d\tau}=\epsilon^{-1}d\Psi(u)V(u)+P(v).
  \label{pkdv-vv}
  \end{equation}
  Here $V(u)=-u_{xxx}+6uu_x$ is the vector filed of KdV and $P(v)$ is the perturbation term,   written in the  $v$-variables. In the action-angle variables $(I,\varphi)$ this equation reads: \begin{equation}
 \frac{d I}{d\tau}=F(I,\varphi),\quad \frac{d\varphi}{d\tau}=\epsilon^{-1}W(I)+G(I,\varphi).
 \label{pkdv-aa}
 \end{equation}
 Here $I\in\mathbb{R}^{\infty}$ and  $\varphi\in\mathbb{T}^{\infty}$, where  $\mathbb{T}^{\infty}:=\{\theta=(\theta_j)_{j\geqslant1},\theta_j\in\mathbb{T}\}$ is the infinite-dimensional torus, endowed with the Tikhonov topology. The two functions $F(I,\varphi)$ and $G(I,\varphi)$ represent  the perturbation term $f$,  written in the action-angle variables, see below (\ref{i1.3}) and (\ref{angle1}).  We consider an averaged equation for the actions:
 \begin{equation}
 \frac{d J}{d\tau}= \langle F\rangle (J), \quad \langle F\rangle (J)=\int_{\mathbb{T}^{\infty}}F(J,\varphi)d\varphi,
 \label{pkdv-av}
 \end{equation}
 where $d\varphi$  is the Haar measure on $\mathbb{T}^{\infty}$.  It turns out that $\langle F\rangle (J)$ defines a Lipschitz vector filed in $h^p_I$ (see (\ref{l1}) below).  So equation (\ref{pkdv-av}) is well-posed, at least locally. We want to study the relation between the actions $I(\tau)$ of solutions for equation~(\ref{pkdv-aa}) and solutions $J(\tau)$ of equation (\ref{pkdv-av}),  for $\tau\in[0,T]$.

 Let $S_{\epsilon}^{\tau}$, $0
 \leqslant\tau\leqslant T$, be the flow-maps of equation (\ref{pkdv-vv}) on $h^p$   and denote
 \[B^v_p(M)=\big\{v\in h^p: \;|v|_p\leqslant M\big\}.\]
 
 \begin{definition}1)  A measure $\mu$ on $h^p$ is called regular if for any analytic function $g$ on $h^p$ such that $g \not\equiv0$, we have $\mu(\{v\in h^p: g(v)=0\})=0$.
 
 2) A measure $\mu$ on $h^p$ is said to be $\epsilon$-quasi-invariant for equation (\ref{pkdv-vv}) on the ball $B_p^v(M)$ if  it is regular, $0<\mu(B_p^v(M))<\infty$  and there exists a constant~$C(T,M)$ such that for  any Borel set $A\subset B^v_p(M)$, we have
 \begin{equation}e^{-\tau C(T,M)}\mu(A)\leqslant \mu(S_{\epsilon}^{\tau}(A))\leqslant e^{\tau C(T,M)}\mu(A).
 \label{quasi-invariant}
 \end{equation}

 \end{definition}
Similarly, these definitions can be carried to measures on the space $H^p$ and the flow-maps of equation (\ref{pkdv1}) on $H^p$.

 The main result of this paper is the  following theorem, in which $u^{\epsilon}(t)$ denotes solutions for equation (\ref{pkdv1}), $v^{\epsilon}(\tau)=\Psi\big(u^{\epsilon}(\epsilon^{-1}\tau)\big)$ denotes solutions for (\ref{pkdv-vv}) and $I(v^{\epsilon}),\varphi(v^{\epsilon})$ are their action-angle variables. By Assumption $A$,  for $\tau\in[0,T]$,
 \[|I(v^{\epsilon}(\tau))|_p^{\sim}\leqslant C_1(|I(v^{\epsilon}(0))|^{\sim}_p).\]
 
 \begin{theorem} Fix any $M>0$. 
 Suppose that  assumption A  holds and equation~(\ref{pkdv1}) has an  $\epsilon$-quasi-invariant  measure $\mu$  on $B^v_{p}(M)$. Then
 
  (i) For any $\rho>0$ and any $q<p+\frac{1}{2}\min\{\zeta_0,1\}$, there exists $\delta_{\rho}>0$, $\epsilon_{\rho,q}>0$ and a Borel subset $\Gamma_{\rho,q}^{\epsilon}\subset B^v_p(M)$ such that
   \begin{equation}\lim_{\epsilon\to0}\mu(B^v_p(M)\setminus \Gamma_{\rho,q}^{\epsilon})=0,
   \label{measure0}
   \end{equation} and for $\epsilon\leqslant \epsilon_{\rho,q}$, we have that if $ v^{\epsilon}(0)\in \Gamma_{\rho,q}^{\epsilon}$,  then
  \begin{equation}
  |I(v^{\epsilon}(\tau))-J(\tau)|_q^{\sim}\leqslant \rho, \quad\mbox{for} \quad0\leqslant \tau\leqslant \min\{T,\mathcal{T}(I^{\epsilon}_0)\}.
 \label{main1}
\end{equation}
Here $I^{\epsilon}_0=I(v^{\epsilon}(0))$, $J(\cdot)$ is the unique solution of the averaged equation (\ref{pkdv-av}) with any  initial data $J_0\in h^p_I$, satisfying $|J_0-I^{\epsilon}_0|_q^{\sim}\leqslant \delta_{\rho}$,  and $$\mathcal{T}(I^{\epsilon}_0)=\min\big\{\tau: |J(\tau)|^{\sim}_p\geqslant C_1(|I^{\epsilon}_0|^{\sim}_p)+1\big\}.$$

(ii) Let $\lambda_{\epsilon}^{v_0}$ be the probability measure on $\mathbb{T}^{\infty}$, defined by the relation 
$$
\int_{\mathbb{T}^{\infty}}f(\varphi)\,d\lambda_{\epsilon}^{v_0}(d\varphi)=
 \frac{1}{T}\int_0^Tf(\varphi(v^{\epsilon}(\tau))d\tau, 
 \quad \forall f\in C(\mathbb{T}^{\infty}),
$$
where $v_0=v^{\epsilon}(0)\in B_p(M)$.
Then the averaged  measure 
$$ \lambda_{\epsilon}:=\frac{1}{\mu\big(B_p(M)\big)}\int_{B_p(M)}   \lambda_{\epsilon}^{v_0}d\mu(v_0 ) $$
  converges weakly, as $\epsilon\to0$,  to the Haar measure $d\varphi$ on $\mathbb{T}^{\infty}$ 
 \end{theorem}
 
 \begin{remark}

 1) Assume that an  $\epsilon$-quasi-invariant measure $\mu$ depends on $\epsilon$, i.e. $\mu=\mu_{\epsilon}$. Then the same conclusion holds with $\mu$ replaced by $\mu_{\epsilon}$, if $\mu_{\epsilon}$ satisfies some consistency conditions.  See subsection 3.3.
 
 2) Item (ii) of Assumption A may be removed if the perturbation is hamiltonian.   See the end of subsection 3.1.
 
 \end{remark}

 Toward  the existence of $\epsilon$-quasi-invariant measures, following \cite{hg2013}, consider a class of  Gaussian measures $\mu_0$  on the Hilbert space $h^p$:
 \begin{equation}
 \mu_0:=\prod_{j=1}^{\infty}\frac{(2\pi j)^{1+2p}}{2\pi\sigma_j}\exp\{-\frac{(2\pi j)^{1+2p}|\mathbf{v}_j|^2}{2\sigma_j}\}d\mathbf{v}_j,
 \label{gaussian}
 \end{equation}
where $d\mathbf{v}_j$, $j\geqslant1$, is the Lebesgue measure on $\mathbb{R}^2$.  We recall that (\ref{gaussian}) is a well-defined probability measure on $h^p$  if and only if $\sum\sigma_j<\infty$(see \cite{Bogachev}). It is regular in the sense of Definition 0.1 and is  non-degenerated in the sense that its support equals to $h^p$ (see \cite{Bogachev, bom2013}). From (\ref{iw-form}), it is easy to see that this kind of measures are invariant for KdV.

 For any $\zeta_0^{\prime}>1$, we say  the  measure $\mu_0$ is \mbox{\it $\zeta_0^{\prime}$-admissible} if the $\sigma_j$  in (\ref{gaussian}) satisfies $0<j^{-\zeta_0^{\prime}}/\sigma_j<const$ for all $j\in\mathbb{N}$. It was proved in \cite{hg2013} that if Assumption A holds and\\
 {\it $(ii)^{\prime}$  the operator  defined by $v\mapsto P(v)$ (see (\ref{pkdv-vv}))  analytically maps the space $h^p$ to the space $h^{p+\zeta_0^{\prime}}$ with some $\zeta_0^{\prime}>1$,}\\  then every $\zeta_0^{\prime}$-admissible measure  $\mu_0$  is $\epsilon$-quasi-invariant for equation (\ref{pkdv1}) on $h^p$.

However,  the conditions $(ii)^{\prime}$ is not easy to verify  due to the   complexity of the nonlinear Fourier transform.  Fortunately, there exists another series of Gibbs-type  measures (see (\ref{gibbs-m}) below) known  to be invariant for KdV, explicitly constructed on the space $H^p$ in \cite{zhidkov}.  We will show in Section 4 that these measures are $\epsilon$-quasi-invariant for equation (\ref{pkdv1}) if Assumption~A holds with $\zeta_0\geqslant 2$.
We point out straight away that this condition is not optimal (see Remark 4.11).

 The paper is organized as follows: Section 1 is about  some important properties of the nonlinear Fourier transform and  the action-angle form of the perturbed KdV~(\ref{pkdv1}). We discuss the averaged equation in Section~2. The Theorem~0.2 is proved in Section 3. Finally we will discuss the existence of $\epsilon$-quasi-invariant measures in Section 4. 
 
 \noindent{\bf Agreements.} Analyticity of maps $B_1\rightarrow B_2$ between Banach spaces $B_1$ and $B_2$, which are the real parts of complex spaces $B_1^c$ and $B_2^c$, is understood in the sense of Fr\'echet. All analytic maps that we consider possess the following additional property: for any $R$, a map  extends to a bounded analytical mapping in a complex ($\delta_R>0$)-neighborhood of the ball $\{|u|_{B_1}<R\}$ in $B_1^c$. We call such analytic maps uniformly analytic.

 \section{Preliminaries on the KdV equation}
     In this section we discuss integrability of the KdV equation (0.1)$_{\epsilon=0}$.

     \subsection{ Nonlinear Fourier transform for KdV}

\begin{theorem}  (see \cite{kjp2003}) There exists an analytic diffeomorphism $\Psi:H^0 \mapsto h^0$ and an analytic functional $K$ on $h^1$ of the form $K(v)=\tilde{K}(I(v))$, where the function $\tilde{K}(I)$ is analytic in a suitable neighborhood of the octant $h_{I+}^1$ in  $h_I^1$, with the following properties:
 
 (i) For any $p\in[-1,+\infty)$, the mapping $\Psi $ defines an analytic diffeomorphism $\Psi : H^p \mapsto h^p$. 
  
 (ii) The differential  $d\Psi(0)$ is the operator  $\sum u_se_s\mapsto v,v_s=|2\pi s|^{-1/2}u_s$.
 
(iii) A curve $u\in C^1(0,T;H^0)$ is a solution of the KdV equation (\ref{pkdv1})$_{\epsilon=0}$ if and only if $v(t)=\Psi(u(t))$ satisfies the equation
  
\begin{equation} \dot{\mathbf{v}}_j=\left(\begin{array}{cc} 0 & -1\\1& 0 \end{array}\right)\frac{\partial \tilde{K}}{\partial I_j}(I)\mathbf{v}_j  , \quad\mathbf{v}_j=(v_j,v_{-j})^t\in\mathbb{R}^2,\;j\in \mathbb{N}.
 \end{equation}
 \label{th-nf}
 \end{theorem}
 
 The coordinates $ v=\Psi(u) $ are called the \emph{Birkhoff coordinates}, and the form (1.1) of KdV is its \emph{Birkhoff normal form}
 
Since the maps $\Psi$ and $\Psi^{-1}$ are analytic, then for $m=0,1,2\dots $, we have
\[||d^j\Psi(u)||_m\leqslant P_m(||u||_m),\quad
 ||d^j\Psi^{-1}(v)||_m\leqslant Q_m(|v|_m ), \quad j=0,1,2,
\]
where $P_m$ and $Q_m$ are continuous functions.    
 
 A remarkable property of the nonlinear Fourier transform $\Psi$ is its quasi-linearity. It means:
 \begin{theorem}(see \cite{kpe2010,kst2010}) 
 If $p\geqslant0$, then the map $\Psi-d\Psi(0):$ $H^p\to h^{p+1}$ is analytic. 
\label{quasi-thm}
 \end{theorem}

 We denote  \[W(I)=(W_1, W_2,\dots),\; \; W_k(I)=\frac{\partial \tilde{K}}{\partial I_k}(I),\;\;   k=1,2,\dots .\] 
This is the KdV-frequency map. It is non-degenerate: 
\begin{lemma} (see \cite{kuk2000}, appendix 6) For any $n\in \mathbb{	N}$, if $ I_{n+1}=I_{n+2}=\cdots=0$, then
\[\mbox{det}\Big((\frac{\partial W_i}{\partial I_j})_{1\leqslant i,j\leqslant n}\Big)\neq 0.\]
\label{lm-nd}
\end{lemma}

Let $l^{\infty}_{-1}$ be the Banach space of all real sequences $l=(l_1,l_2,\dots)$ with the norm
 \[|l|_{-1}=\sup_{n\geqslant 1}n^{-1}|l_n|<\infty.\]
 Denote $\boldsymbol{\kappa}=(\kappa_n)_{n\geqslant 1}$, where
 $\kappa_n=(2\pi n)^3$.
 \smallskip
 
 \begin{lemma} (see \cite{kjp2003}, Thoerem 15.4)  The normalized frequency map
\[ I\mapsto W(I)-\boldsymbol{\kappa} \]
 is real analytic as a map from $h^1_{I+}$ to $l^{\infty}_{-1}$.  
 \label{lm-quasi1}
 \end{lemma}

 \subsection{ Equation (\ref{pkdv1}) in the Birkhoff coordinates.}
 For $k=1,2\dots $  we denote:\[\Psi_k : H^m \to \mathbb{R}^2,\quad \Psi_k(u)=\mathbf{v}_k,\]
 where  $\Psi(u)=v=(\mathbf{v}_1,\mathbf{v}_2,\dots)$.
  Let $u(t)$ be a solution of equation (\ref{pkdv1}). Passing to the slow time $\tau=\epsilon t$ and denoting $\dot{}$ to be $\frac{d}{d\tau}$,  we get
 \begin{equation}
 \dot{\mathbf{v}}_k=d\Psi_k(u)(\epsilon^{-1}V(u)) +\mathcal{P}_k(v), \quad k\geqslant 1,\label{tvpkdv}
 \end{equation}
 where $V(u)=-u_{xxx}+6uu_x$ and $\mathcal{P}_k(v)=d\Psi_{k}(\Psi^{-1}(v))(f(\Psi^{-1}(v)))$.    Since the action \mbox{$I_k(v)=\frac{1}{2}|\Psi_k|^2$} is an integral of motion for the KdV equation (\ref{pkdv1})$_{\epsilon=0}$, we have 
 \begin{equation}\dot{I}_k= (\mathcal{P}_k(v),\mathbf{v}_k):= F_k(v).\label{i1.3}
 \end{equation}
Here and below $(\cdot,\cdot)$ indicates the scalar product in $\mathbb{R}^2$. 

For $k\geqslant 1$ defines the angle $\varphi_k=\arctan(\frac{v_{-k}}{v_k})$ if $\mathbf{v}_k\neq 0$ and $\varphi_k=0$ if $\mathbf{v}_k=0$.    Using equation (\ref{tvpkdv}), we get
\begin{equation}\dot{\varphi}_k=\epsilon^{-1}W_k(I)+|\mathbf{v}_k|^{-2}(d\Psi_k(u)f(x,u),\mathbf{v}_k^{\bot}),\quad \mbox{if} \quad\mathbf{v}_k\neq 0,\label{angle1}
\end{equation}
  where $\mathbf{v}_k^{\bot}=(-v_{-k},v_k) $. Denoting for brevity, the vector field in equation (1.4) by $\epsilon^{-1}W_k(I)+G_k(v)$, we rewrite the equation for the pair $(I_k, \varphi_k) (k\geqslant1)$ as
 \begin{equation}
 \begin{split}
& \dot{I}_k= F_k(v) = F_k(I,\varphi), \\
 &\dot{\varphi}_k=\epsilon^{-1}W_k(I)+ G_k(v). 
 \end{split}
 \label{iapkdv}
 \end{equation}
 
  We set \[F(I,\varphi)=(F_1(I,\varphi),F_2(I,\varphi),\dots).\]
 Denote $$\bar{\zeta}_0=\min\{1,\zeta_0\}.$$
  For any $q\in[p^{\prime},p]$,  define  a map  $\mathcal{P}$ as
 \[\mathcal{P}: \;h^q\to h^{q+\bar{\zeta}_0}, \quad v\mapsto (\mathcal{P}_1(v),\dots).\]
Clearly, $\mathcal{P}(v)=d\Psi\big(\Psi^{-1}(v)\big)\big(f(\Psi^{-1}v)\big)$. Then Theorem \ref{quasi-thm} and Assumption A imply that the map $\mathcal{P}$ is analytic.
   Using (\ref{i1.3}),  for any $k\in\mathbb{N}$,  we have 

\[(2\pi k)^{2q+1+\bar{\zeta}_0} |F_k(v)|\leqslant (2\pi k)^{2q+1}|\mathbf{v}_k|^2+(2\pi k)^{2q+1+2\bar{\zeta}_0}|\mathcal{P}_k(v)|^2.\]
Therefore,
\begin{equation}
|F(I,\varphi)|_{q+\bar{\zeta}_0/2}^{\sim}\leqslant |v|^2_q+|\mathcal{P}(v)|^2_{q+\bar{\zeta}_0}\leqslant C(|v|_q).
\label{b-f1}
\end{equation}

  In the following lemma $P_k$ and $P_k^j$ are some fixed continuous functions.
 
 \begin{lemma}For $k,  j\in \mathbb{N}$ and each  $q\in[p^{\prime}, p],$ we have:
 
 (i) The function $F_k(v)$ is analytic in each space $h^{q}$.
 
 (ii) For any  $\delta>0$, the function $G_k(v)\chi_{ \{I_k \geqslant \delta\}}$ is bounded by $\delta^{-1/2}P_k(|v|_{q})$.
 
 (iii) For any $\delta>0$, the function $\frac{\partial F_k}{\partial I_j}(I,\varphi)\chi_{\{I_j\geqslant \delta\} }$ is bounded by $\delta^{-1/2}P_k^j(|v|_{q})$.
 
 (iv) The function $\frac{\partial F_k}{\partial \varphi_j}(I,\varphi)$ is bounded by $P_k^j(|v|_{q})$,  and  for any $n\in \mathbb{N}$ and  $(I_1,\dots,I_n)\in \mathbb{R}^n_+$, the function  $\varphi\mapsto F_k(I_1,\varphi_1,\dots, I_n,\varphi_n,0,\dots)$ is smooth on $\mathbb{T}^n$.
 \label{lm-regular}
 
  \end{lemma}

   We  denote
   \[\begin{split}&\Pi_I: h^p\to h^p_I,\quad \Pi_I(v)=I(v),\\
   &\Pi_{I,\varphi}:h^p\to h_I^p\times \mathbb{T}^{\infty},\quad \Pi_{I,\varphi}(v)=(I(v),\varphi(v)).
   \end{split}
   \]
   Abusing notation, we will identify $v$ with $(I,\varphi)=\Pi_{I,\varphi}(v)$.
   \smallskip
   
   \begin{definition}  We say that a curve   $\big(I(\tau),\varphi(\tau)\big)$, $\tau\in[0, \tau_1] $, is a regular solution of equation (\ref{iapkdv}),   if there  exists a solution $u(\cdot)\in H^{p}$ of equation (\ref{pkdv1})  such that 
    \[\Pi_{I,\varphi}\big(\Psi(u(\epsilon^{-1}\tau))\big)=(I(\tau),\varphi(\tau)),\quad \tau\in[0, \tau_1].\]
  \end{definition}

Note that if $(I(\tau),\varphi(\tau))$ is a regular solution, then each $I_j(\tau)$ is a $C^1$-function, while $\varphi_j(\tau)$ may be discontinuous at points $\tau$, where $I_j(\tau)=0$.

\section{ Averaged equation}

     For a function $f$ on a Hilbert space $H$, we write $f\in Lip_{loc}(H)$ if
      \begin{equation} |f(u_1)-f(u_2)|\leqslant P(R)||u_1-u_2||, \quad \text{if}\quad ||u_1||, ||u_2||\leqslant R, \label{2.1}
      \end{equation}
    for a suitable continuous function $P$ which depends on $f$.
       By  the Cauchy inequality, any analytic function on $H$ belongs to $Lip_{loc}(H)$ (see Agreements).  In particularly, for any $k\geqslant 1$,
     \begin{equation}W_k(I)\in Lip_{loc}(h_I^q),\;\;   q\geqslant 1,\quad \mbox{and} \quad
      F_k(v)\in Lip_{loc}(h^q),\;\; q\in[p^{\prime}(p),p].
      \label{l2.1}
      \end{equation}
      
         Let $f\in Lip_{loc}(h^{p_0})$ for some $p_0\geqslant 0$ and $v\in h^{p_1}$,   $ p_1>p_0$. Denoting by $\Pi^M$,  $M\geqslant 1$, the projection
     \[ \Pi^M: h^0 \rightarrow h^0, \quad (\mathbf{v}_1,\mathbf{v}_2,\dots) \mapsto (\mathbf{v}_1,\dots,\mathbf{v}_M, 0,\dots),\]
     we have $|v-\Pi^M v|_{p_0}\leqslant (2\pi M)^{-(p_1-p_0)}|v|_{p_1}$.
     Accordingly, 
     \begin{equation}
     |f(v)-f(\Pi^M v)|\leqslant P(|v|_{p_1})(2\pi M)^{-(p_1-p_0)}.\label{lip}
     \end{equation}
     
     The torus $\mathbb{T}^{\infty}$ acts on the space $ h^0$ by the linear transformations $\Phi_{\theta}$, $\theta\in \mathbb{T}^{\infty}$, where $\Phi_{\theta}: (I, \varphi) \mapsto(I,\varphi +\theta)$.     
     For a function $f\in Lip_{loc}(h^{p})$,  we define the averaging in  angles as \[\langle f\rangle(v)=\int_{\mathbb{T}^{\infty}}f(\Phi_{\theta}(v))d\theta,\]
     where  $d\theta$ is the Haar measure on $\mathbb{T}^{\infty}$.  
     Clearly, the average $\langle f\rangle$ is independent of $\varphi$. Thus $\langle f\rangle$ can be written as $\langle f\rangle(I)$.
     
    Extend the mapping $\pi_I$ to a complex mapping $h^p\bigotimes\mathbb{C}\to h^p_I\bigotimes\mathbb{C}$, using the same formulas (\ref{maps}).
     Obviously, if $\mathcal{O}$ is a complex neighbourhood of $h^p$, then $\pi^c_I(\mathcal{O})$ is a complex neighbourhood of $h_I^p$.

   \begin{lemma}  (See \cite{kpa2008}, Lemma 4.2) Let $f\in Lip_{loc}(h^{p})$, then
     
     (i) The function  $\langle f\rangle(v)$ satisfy (\ref{2.1}) with the same function $P$ as $f$ and take the same value at the origin.
     
     (ii) This function  is smooth (analytic) if $f$ is.  If $f(v)$ is analytic in a complex  neighbourhood $\mathcal{O}$ of  $h^{p}$, then $\langle f\rangle(I)$ is analytic in the complex neighbourhood $\pi_I^c(\mathcal{O})$ of $h^p_I$.
          \label{f-average}
     \end{lemma}

    For any $\bar{q}\in[p^{\prime},p]$, we consider the mapping defined by 
    \[\langle F\rangle: h_I^{\bar{q}}\to h_I^{\bar{q}+\bar{\zeta}_0/2},\quad J\mapsto \langle F\rangle (J),\]
    where $\langle F\rangle(J)=(\langle F_1\rangle(J),\langle F_2\rangle(J),\dots)$.
    
             \begin{corollary} For every $\bar{q}\in[p^{\prime},p]$, 
              the mapping  $\langle F\rangle$ is analytic  as a map from the space $h^{\bar{q}}_I$ to $h^{\bar{q}+\bar{\zeta}_0/2}_I$.
              \end{corollary}
\begin{proof} 
The mapping $\mathcal{P}(v)$ extends  analytically to a complex neighbourhood $\mathcal{O}$ of $h^{\bar{q}}$ (see Agreements).
Then by (\ref{i1.3}), the functions $F_j(v)$, $j\in\mathbb{N}$  are analytic in $\mathcal{O}$. Hence it follows from Lemma 2.1 that for  each $j\in\mathbb{N}$, the function $\langle F_j\rangle$ is analytic in the complex neighbourhood  $\pi^c_I(\mathcal{O})$ of $h^{\bar{q}}_I$.  By (\ref{b-f1}), the mapping $\langle F\rangle$ is locally bounded on $\pi^c_I(\mathcal{O})$.  It is well known that the analyticity of each coordinate function and the locally boundness  of the maps imply the analyticity of the maps (see, e.g. \cite{analytic}). This finishes the proof of the corollary.
\end{proof}

   We recall that   a vector $\omega\in\mathbb{R}^n$ is  called \mbox{\it non-resonant }  if
   \[\omega\cdot k\neq 0, \quad\forall \;\; k\in \mathbb{Z}^n\setminus\{0\}.\]
   Denote by $C^{0+1}(\mathbb{T}^n)$ the set of all Lipschitz functions on $\mathbb{T}^n$. The following lemma is a version of the classical Weyl theorem (for a proof, see  e.g. Lemma 2.2 in \cite{hg2013}).   
   \begin{lemma} Let $f\in C^{0+1}(\mathbb{T}^n)$ for some $n\in\mathbb{N}$. For any $\delta>0$ and any non-resonant vector $\omega\in\mathbb{R}^n$, there exists $T_0>0$ such that if $T\geqslant T_0$, then 
   \[\Big|\frac{1}{T}\int_0^Tf(x_0+\omega t)dt-\langle f\rangle \Big|\leqslant \delta,
   \]
   uniformly in $x_0\in\mathbb{T}^n$.
   \end{lemma}


 \section{ Proof of the main theorem }
       In this section we prove Theorem 0.2.
       
          Assume  $u(0)=u_0\in H^{p}$.
       So \begin{equation}
       \Pi_{I,\varphi}(\Psi(u_0))=(I_0,\varphi_0)\in h_{I+}^{p}\times \mathbb{T}^{\infty}.
       \end{equation}   
                We denote \[B_p^I(M)=\{I\in h^p_{I+}: |I|^{\sim}_p\leqslant M\}.\]
      Without loss of generality, we  assume that $T=1$.      
                 Fix any  $M_0>0$.
      Let \[(I_0,\varphi_0)\in B^I_p(M_0)\times \mathbb{T}^{\infty}:=\Gamma_0,\] 
      that is,
      \[v_0=\Psi(u_0)\in B_p^v(\sqrt{M_0}).\]
      We pass to the  slow time $\tau=\epsilon t$. 
      Let $(I(\cdot),\varphi(\cdot))$ be a regular solution of the system (\ref{iapkdv}) with $(I(0),\varphi(0))=(I_0,\varphi_0)$.  We will also write it as $(I^{\epsilon}(\cdot),\varphi^{\epsilon}(\cdot))$ when we want to stress the dependence on $\epsilon$. Then by assumption A, there exists $M_1\geqslant M_0$  such that
       \begin{equation}I(\tau)\in B_p^I(M_1),\quad \tau\in [0,1].
       \end{equation}
      
           By (\ref{b-f1}),  we know that 
           \begin{equation}
           |F(I,\varphi)|^{\sim}_1\leqslant C_{M_1},\quad \forall \;(I,\varphi)\in B_p^I(M_1)\times \mathbb{T}^{\infty},\label{boundf}
           \end{equation}
           where the  constant  $C_{M_1}$ depends only on $M_1$.

       We denote $I^m=(I_1,\dots, I_m, 0,0,\dots)$,  $\varphi^m=(\varphi_1,\dots,\varphi_m,0,0,\dots)$, and $W^m(I)=(W_1(I),\dots, W_m(I), 0, 0,\dots)$, for any $m\in \mathbb{N}$.
       
       \subsection{Proof of assertion (i)}
       Fix any  \[n_0\in \mathbb{N} \quad \text{and}\quad  \rho>0.\] By (2.2),     there exists $m_0\in \mathbb{N}$ such that 
     \begin{equation}|F_k(I,\varphi)-F_k(I^{m_0},\varphi^{m_0})|\leqslant \rho,\quad \forall (I,\varphi)\in B^I_p(M_1)\times\mathbb{T}^{\infty},\label{dn}
     \end{equation}
     where $k=1,\cdots, n_0$.

             From now on, we always assume that 
             \[(I,\varphi)\in\Gamma_1:= B_p^I(M_1)\times\mathbb{T}^{\infty},\quad\mbox{i.e.}\quad v\in B_p^v(\sqrt{M_1}),\] 
            and identify $v\in h^{p}$ with $(I,\varphi)=\Pi_{I,\varphi}(v).$
            
              By Lemma \ref{lm-regular}, we have 
             \begin{equation}
             \begin{split}&|G_j(I,\varphi)|\leqslant \frac{C_0(j,M_1)}{\sqrt{I_j}},\\
   &  |\frac{\partial F_k}{\partial I_j}(I,\varphi)|\leqslant \frac{C_0(k,j, M_1)}{\sqrt{I_j}},\\
      &|\frac{\partial F_k}{\partial \varphi_j}(I,\varphi)|\leqslant C_0(k,j,M_1).
      \end{split}
      \label{clumsy1}
         \end{equation}
                     From Lemma \ref{lm-quasi1} and Lemma \ref{f-average}, we know that
                     \begin{equation}
                     \begin{split}
                      &|W_j(I)-W_j(\bar{I})|\leqslant C_1(j,M_1)|I-\bar{I}|_1,\\
                           & |\langle F_k\rangle(I)-\langle F_k \rangle(\bar{I})| \leqslant C_1(k,j,M_1)|I-\bar{I}|_1.
                           \end{split}\label{clumsy2}
                            \end{equation}
                            By (\ref{l2.1}) we get
                            \begin{equation}
                            |F_k(v^{m_0})-F_k(\bar{v}^{m_0})|\leqslant C_2^{\prime}(k,M_1)|v^{m_0}-\bar{v}^{m_0}|_p\leqslant C_2(k,m_0,M_1)|v^{m_0}-\bar{v}^{m_0}|_{\infty},\label{clumsy3}
                            \end{equation}
                            where $|\cdot|_{\infty}$ is the $l^{\infty}$-norm.
                            
                            We denote
                            \[C_{M_1}^{n_0,m_0}=m_0\cdot\max\{C_0,C_1,C_2:1\leqslant j\leqslant m_0,1\leqslant k\leqslant n_0\}.\]
                 Below we define a number of sets, depending on various parameters. All of them also depend on $\rho$, $n_0$ and $m_0$, but this dependence is not indicated.            For any $\delta>0$ and $T_0>0$, we define a subset \[E(\delta,T_0)\subset \Gamma_1\] as the  collection of all $(I,\varphi)\in \Gamma_1$ such that for every $T\geqslant T_0$, we have, 
      \begin{equation}
      \Big|\frac{1}{T}\int_0^{T}[F_k(I^{m_0},\varphi^{m_0}+W^{m_0}(I)s)-\langle F_k\rangle(I^{m_0})]ds\Big|\leqslant\delta,\quad \text{for}\quad k=1,\dots,n_0.
     \label{average}
      \end{equation}
      
      Let $\mathcal{S}_\epsilon^{\tau}$ be the flow generated by regular solutions of the system (1.5). We define two more groups of sets.
      \[\Delta(\tau)=\Delta(\tau,\epsilon,\delta, T_0,I,\varphi):=\{\tau_1\in[0,\tau]:\mathcal{S}_\epsilon^{\tau_1}(I,\varphi)\notin E(\delta,T_0)\}.\]
      \[N(\gamma)=N(\gamma,\epsilon,\delta,T_0):=\{(I,\varphi)\in \Gamma_0: \mbox{Mes}[\Delta(1, \epsilon,\delta, T_0,I,\varphi)]\leqslant \gamma\}.\]
      Here and below $\mbox{Mes}[\cdot]$ stands  for the  Lebesgue measure in $\mathbb{R}$.  We will indicate the dependence of  the set $N(\gamma)$ on $n_0$ and $\rho$ as $N_{n_0,\rho}(\gamma)$,  when necessary.

      By continuity,  $E(\delta,T_0)$ is a closed subset of $\Gamma_1$ and $\Delta(\tau)$ is an open subset of $[0,\tau]$.           
            Repeating a version of the classical averaging argument (cf. \cite{loc1988}), presented   in the proof of Lemma 4.1 in \cite{hg2013}, we have the following averaging lemma:

      \begin{lemma} For $k=1,\dots,n_0$, the $I_k$-component of any regular solution of (\ref{iapkdv}) with initial data in $N(\gamma,\epsilon,\delta,T_0)$ can be written as\[I_k(\tau)=I_k(0)+\int_0^{\tau}\langle F_k\rangle(I(s))ds +\Xi_k(\tau),\]
      where the function $|\Xi_k(\tau)|$ is bounded  on $[0,1]$ by
     \[
    \begin{split}&4C_{M_1}^{n_0,m_0}\Big\{\big[2(\epsilon^{1/4}+2T_0C_{M_1}\epsilon)^{1/2}\big](\epsilon T_0+\epsilon\gamma+1)\\
    &+\big[T_0C_{M_1}\epsilon^{7/8}+T_0C_{M_1}\epsilon +C_{M_1}^{n_0,m_0}(\frac{1}{2}T_0\epsilon^{7/8}+\frac{1}{3}C_{M_1}T_0^2\epsilon)\big](\epsilon T_0+\epsilon\gamma+1)\\
    &+4C_{M_1}\gamma+2\rho+2\delta+2C_{M_1}T_0\epsilon.
    \end{split}
                 \]
             \label{lemma1}
\end{lemma}

 \begin{corollary}  For any $\bar{\rho}>0$, with  a suitable choice of  $\rho$, $\delta$, $T_0$, $\gamma$, the function $|\Xi_k(\tau)|$ in Lemma \ref{lemma1} can be made smaller than $\bar{\rho}$, if $\epsilon$ is small enough.
 \label{lemma2}
      \end{corollary}
      \begin{proof}
      We choose
      \[ T_0=\epsilon^{-\sigma},\; \gamma=\frac{\bar{\rho}}{9C_{M_1}},\;\delta=\rho=\frac{\bar{\rho}}{9}\]
      with $0<\sigma<\frac{1}{2}$. Then for $\epsilon$ sufficiently small we have \[|\Xi_k(\tau)|<\bar{\rho}.\]
      \end{proof}
      Now  let $\mu$ be a regular $\epsilon$-quasi-invariant measure   and $\{S^{\tau}_{\epsilon},\tau\geqslant0\}$ be the flow of equation (\ref{pkdv-vv}) on $h^p$.   Below we follow the arguments, invented by Anosov for the finite dimensional averaging (e.g. see in \cite{loc1988}). 
       
      Consider the measure $\mu_1=d\mu dt$ on $h^p\times \mathbb{R}$.    Define  the following subset of $h^p\times\mathbb{R}$: 
      \[\mathcal{B}^{\epsilon}:=\{(v,\tau): v\in \Gamma_0,\;\; \tau\in[0,1]\;\;\text{and}\;\; S^{\tau}_{\epsilon}v\in\Gamma_1\setminus E(\delta,T_0)\}.\]
      Then by (\ref{quasi-invariant}), there exists $C(M_1)$ such that \[\mu_1(\mathcal{B}^{\epsilon})=\int_0^1\mu\Big(\Gamma_0\cap S^{-\tau}_{\epsilon}\big(\Gamma_1\setminus E(\delta,T_0)\big)\Big)d\tau\leqslant e^{C(M_1)}\mu(\Gamma_1\setminus E(\delta,T_0).\]
      For any $v\in\Gamma_0$, denote $\bar{\Delta}(v)=\Delta(1,\epsilon,\delta,T_0,I,\varphi)$, where $(I,\varphi)=\Pi_{I,\varphi}(v)$. Then by the Fubini theorem, we have 
      \[\mu_1(\mathcal{B}^{\epsilon})=\int_{\Gamma_0}Mes(\bar{\Delta}(v))d\mu(v).\]
       Using Chebyshev's inequality, we obtain
      \begin{equation}\mu\big(N(\gamma,\epsilon,\delta,T_0)\big)\leqslant \frac{e^{C(M_1)}}{\gamma}\mu\big(\Gamma_1\setminus E(\delta,T_0)\big).
      \label{meas1}
      \end{equation} 
      By the definition of $E(\delta,T_0)$, we know that \begin{equation}
      E(\delta,T_0)\subset E(\delta,T_0^{\prime} ),\quad\text{ if}
      \quad  T_0^{\prime}\geqslant T_0.
      \label{sigma}
      \end{equation}
        We set  $E^{\infty}(\delta):=\cup_{T_0\geqslant1}E(\delta,T_0)$. 
      Define 
      \[\mathcal{RES}(m_0)=\Big\{(I,\varphi)\in \Gamma_1: \;\exists k\in \mathbb{Z}^{m_0}\;\;\text{such that}\;\;k_1W_1(I)+\dots+k_{m_0}W_{{m_0}}(I)=0 \Big\}.\]
     Since the measure  $\mu$ is regular, then by Lemma \ref{lm-nd}, we have that   \mbox{$\mu(\mathcal{RES}(m_0))=0$}.  If $(I^{\prime},\varphi^{\prime})\in \Gamma_1\setminus\mathcal{RES}(m_0)$, then the vector $W^{m_0}(I^{\prime})\in\mathbb{R}^{m_0}$ is non-resonant. Due to Lemma 2.3, we know that there exists $T_0^{\prime}>0$ such that for $T\geqslant T_0^{\prime}$, the inequalities (\ref{average}) hold. Therefore $(I^{\prime},\varphi^{\prime})\in E(\delta,T_0^{\prime})\subset E^{\infty}(\delta)$. Hence\[ \Gamma_1\setminus E^{\infty}(\delta)\subset \mathcal{RES}(m_0).\]
     So we have $\mu(E^{\infty}(\delta))=\mu(\Gamma_1).$
     Since $\mu(E^{\infty}(\delta))=\lim_{T_0\to\infty}E(\delta,T_0)$ due to (\ref{sigma}), then for any $\nu>0$, there exists $T_0^{\prime}>0$ such that for each $T_0\geqslant T_0^{\prime}$, we have \[\mu(E^{\infty}\setminus E(\delta,T_0))\leqslant \nu.\]
     So the r.h.s of the inequality  (\ref{meas1}) can be made  arbitrary small if $T_0$ is large enough. 
      
      Fix some $0<\sigma<1/2$, we have proved the following lemma.

      \begin{lemma} Fix any $\delta>0$, $\bar{\rho}>0$.  Then for every $\nu>0$ we can find $\epsilon(\nu)>0$ such that, if $\epsilon <\epsilon(\nu)$, then  \[\mu\Big(\Gamma_0\setminus N(\frac{\bar{\rho}}{9C_{M_1}})\Big)< \nu,\] where $N(\frac{\bar{\rho}}{9C_{M_1}})=N(\frac{\bar{\rho}}{9C_{M_1}},\epsilon,\delta, \epsilon^{-\sigma})$.
      \label{lm-small}
\end{lemma}

            We now are  in a position to prove assertion (i) of Theorem 0.2.
            
             By Corollary 2.2, for each $q\in[p^{\prime}, p]$, there exists  $C_3(q, M_1)$ such that for any $J_1,J_2\in B^{I}_{\bar{q}}(M_1+1)$ (see Agreements),
            \begin{equation}|\langle F\rangle(J_1)-\langle F\rangle(J_2)|^{\sim}_{q}\leqslant |\langle F\rangle(J_1)-\langle F\rangle(J_2)|^{\sim}_{q+\bar{\zeta}_0/2}\leqslant C_3(\bar{q}, M_1)|J_1-J_2|^{\sim}_{q}.\label{l1}
            \end{equation}
            Since the mapping $\langle F\rangle: h^p_I\to h^p_I$ is locally Lipschitz by (\ref{l1}),  then using Picard's theorem, for any $J_0\in B^I_p(M_1)$ there exists  a unique solution $J(\cdot)$ of the averaged equation (\ref{pkdv-av}) with $J(0)=J_0$. 
                We denote \[\mathcal{T}(J_0):=\inf\{\tau>0: |J(\tau)|_p>M_1+1\}\leqslant \infty.\]

                   For any $\bar{\rho}>0$ and $q<p+\zeta_0$,  there exist $n_1\in\mathbb{N}$ such that 
            \begin{equation}
            \begin{split}&|F(I,\varphi)-F^{n_1}(I,\varphi)|_q<\frac{\bar{\rho}}{8}e^{-C_3(M_1)},\quad (I,\varphi)\in B^I_p(M_1+1)\times\mathbb{T}^{\infty},\\
            &|\langle F\rangle(J)-\langle F\rangle^{n_1}(J)|_q<\frac{\bar{\rho}}{8}e^{-C_3(M_1)},\quad J\in B^I_p(M_1+1).\label{l2}
            \end{split}
            \end{equation}
            Here \[C_3(M_1)=\begin{cases} C_3(p,M_1)\quad &\text{if}\quad q>p,\\
            C_3(q,M_1)\quad&\text{if}\quad q\leqslant p.
            \end{cases}
            \]
            Find  $\rho_0$ from the relation
           
            \[8\sum_{j=1}^{n_1}j^{1+2q}\rho_0=\bar{\rho}e^{-C_3(M_1)}.\]
            
            By Lemma 3.1 and Corollary 3.2,   there exists $\epsilon_{\bar{\rho},q}$ such that if $\epsilon\leqslant \epsilon_{\bar{\rho},q}$, then for initial data in the subset  $\Gamma_{\bar{\rho}}=N_{n_1,\rho_0}(\frac{\rho_0}{9C_{M_1}\epsilon},\epsilon,\frac{\rho_0}{9},\epsilon^{-\sigma})$ we have       for $k=1,\cdots, n_1$,
            \begin{equation}
            I^{\epsilon}_k(\tau)=I^{\epsilon}_k(0)+\int_0^\tau \langle F_k\rangle (I^{\epsilon}(s))ds+\xi_k(\tau),\quad |\xi_k(\tau)|<\rho_0,\quad \tau\in[0,1],
                       \label{i-av1}
                       \end{equation}
            Therefore, by (\ref{l2}) and (\ref{i-av1}), for $(I^{\epsilon}(0),\varphi^{\epsilon}(0))\in \Gamma_{\bar{\rho}}$, $J(0)\in B_p(M_1+1)$ and \mbox{$|\tau|\leqslant \min\{1,\mathcal{T}(J(0))\}$},
            \[
            \begin{split}&| I^{\epsilon}(\tau)-J(\tau)|_q^{\sim}-|I^{\epsilon}(0)-J(0)|_q^{\sim}\\
            &\leqslant \int_0^{\tau}|F(I^{\epsilon}(s))ds-\int_0^{\tau}\langle F\rangle(J(s))|_q^{\sim}ds\\
            &\leqslant\int_0^{\tau}|F^{n_1}(I^{\epsilon}(s))-\langle F\rangle^{n_1}(J(s))|_q^{\sim}ds+\frac{\rho}{4}e^{-C_3(M_1)}.\\
            &\leqslant \int_0^{\tau}|\langle F\rangle(I^{\epsilon}(s))-\langle F\rangle(J(s))|_q^{\sim}ds+\frac{\rho}{2}e^{-C_3(M_1)}.
            \end{split}\]
            Using (\ref{l1}), we obtain
            \[|I^{\epsilon}(\tau)-J(\tau)|^{\sim}_q\leqslant |I^{\epsilon}(0)-J(0)|^{\sim}_q +\int_0^{\tau} C_3(M_1)|I^{\epsilon}(s)-J(s)|^{\sim}_q ds+\xi_0(\tau),\label{small-1}\]
            where $|\xi_0(\tau)|\leqslant \frac{\bar{\rho}}{2}e^{-C_3(M_1)}$. 
             By Gronwall's lemma,  if \[|I^{\epsilon}(0)-J(0)|_q^{\sim}\leqslant \delta=e^{-C_3(M_1)}\bar{\rho},\] then 
           \[ |I(\tau)-J(\tau)|_q\leqslant 2\bar{\rho},\quad |\tau|\leqslant \min\{1,\mathcal{T}(J(0))\}.\]
                      This establishes inequality (\ref{main1}).
           Assuming  that $\bar{\rho}<<1$, we get from the definition of $\mathcal{T}(J(0))$ that $\mathcal{T}(J(0))$  is bigger than 1, if $\zeta_0>0$ and $q\geqslant p$.
 From Lemma 3.3  we know that for any $\nu>0$,  if $\epsilon$ small enough, then $\mu(\Gamma_0\setminus\Gamma_{\bar{\rho}})<\nu$.  This completes the proof of  the assertion (i) of Theorem 0.2.
 \smallskip
 
 {\it Proof of statement (2) of Remark 0.3}\quad
 If the perturbation is hamiltonian with Hamiltonian $H$,  then $F=-\nabla_{\varphi}H$.  Therefore the averaged vector filed $\langle F\rangle=0$.
 For any $\rho>0$ and any  $q<p$, there exists $n_2$ such that 
 \[|I-I^{n_2}|_q^{\sim}<\rho/4,\quad \forall I\in B_p(M).\]
 By similar argument, we can obtain that, there exists a subset $\Gamma_{\rho,n_2}^{\epsilon}\subset\Gamma_0$, satisfying~(\ref{measure0}), such that  for initial data $(I^{\epsilon}(0),\varphi^{\epsilon}(0))\in\Gamma_{\rho,n_2}^{\epsilon}$, 
   and for $\tau\in[0,1]$, we have
 \[|I^{\epsilon,n_2}(\tau)-I^{\epsilon,n_2}(0)|_q^{\sim}\leqslant \rho/4.\]
 So 
 \[|I^{\epsilon}(\tau)-I^{\epsilon}(0)|_q^{\sim}\leqslant \rho\quad \text{for}\quad \tau\in[0,1].\]
 In this argument we do not require $\zeta_0\geqslant0$.  So item (ii) of Assumption A is not needed if the perturbation is hamiltonian.

 \subsection{Proof of the assertion (ii)}
 
     We fix $\alpha<1/4$. For any $(m,n)\in \mathbb{N}^2$ denote 
            \[
            \begin{split}&\mathcal{B}_m(\epsilon):=\Big\{( I,\varphi)\in \Gamma_1: \inf_{k\leqslant m}|I_k|<\epsilon^{\alpha}\Big\},\\
            &\mathcal{R}_{m,n}(\epsilon):=\bigcup_{|L|\leqslant n,L\in\mathbb{Z}^m\setminus\{0\}}\Big\{(I,\varphi)\in \Gamma_1:| W(I)\cdot L|<\epsilon^{\alpha}\}\Big\}.
            \end{split}
            \]
            Then let
            \begin{equation}\Upsilon_{m,n}(\epsilon)=\Big(\bigcup_{m_0\leqslant m}\mathcal{R}_{m_0,n}(\epsilon)\Big)\cup\mathcal{B}_m(\epsilon),
            \label{b_m}
            \end{equation}
           and for any $(I_0,\varphi_0)\in\Gamma_0$, denote 
            \[S(\epsilon, m, n, I_0,\varphi_0)=\{\tau\in[0,1]: (I^{\epsilon}(\tau),\varphi^{\epsilon}(\tau))\in\Upsilon_{m,n}(\epsilon)\}，\]
           
           Fix  $m\in\mathbb{N}$, take a bounded Lipschitz function $g$ defined on the torus $\mathbb{T}^m$ such that $Lip(g)\leqslant 1$ and $|g|_{L_{\infty}}\leqslant 1$. Let
           $\sum_{s\in\mathbb{Z}^{m}}g_s e^{i s\cdot \varphi}$ be  its Fourier series.  Then for any $\rho>0$, there exists $n$, such that if we denote $\bar{g}_n=\sum_{|s|\leqslant n}g_s e^{i s\cdot\varphi}$, then
           \[\Big|g(\varphi)-\bar{g}_n(\varphi)\Big|<\frac{\rho}{2},\quad\forall \varphi\in\mathbb{T}^{m}.\]
           As the measure $\mu$ is regular and $\Upsilon_{m,n}(\epsilon_1)\subset\Upsilon_{m,n}(\epsilon_2)$ if $\epsilon_1\leqslant \epsilon_2$, then 
           \[\mu(\Upsilon_{m,n}(\epsilon))\to0, \quad\epsilon\to0.\]
           Since the measure $\mu$ is $\epsilon$-quasi-invariant,  then following the same argument that proves Lemma 3.3, we have  there exists subset $\Lambda_{\rho}^{\epsilon}\subset \Gamma_0$ such that   for initial data $(I_0,\varphi_0)\in  \Lambda_{\rho}^{\epsilon}$, we have $\text{Mes}\big(S(\epsilon, m,n, I_0,\varphi_0   )\big)\leqslant \rho/4$, and if $\epsilon$ is small enough, then $\mu(\Gamma_0\setminus\Lambda_{\rho}^{\epsilon})<\mu(\Gamma_0)\rho/4$. Due to Lemma 2.3, if $(I^{\epsilon}(\cdot),\varphi^{\epsilon}(\cdot))$ stays  long  enough time outside the subset $\Upsilon_{m,n}(\epsilon)$,  then the time average of $\bar{g}(\varphi^{\epsilon,m}(\tau))$ can be well approximated by its space average.  Following an argument of Anosov (see \cite{loc1988}), 
      we  obtain  that
for $\epsilon$ small enough and for initial data $(I_0,\varphi_0)\in\Lambda_{\rho}^{\epsilon}$, we have 
\begin{equation}\Big|\int_{\mathbb{T}^m}\bar{g}(\varphi)d\lambda_{\epsilon}^{I_0,\varphi_0}-\int_{\mathbb{T}^m}\bar{g}d\varphi\Big|=\Big|\int_0^1\bar{g}\big(\varphi^{\epsilon, m}(\tau)\big)d\tau-\int_{\mathbb{T}^m}\bar{g}(\varphi)d\varphi\Big|<\rho/2.
\label{rho-small}
\end{equation}
(For a proof, see Lemma 4.1 in \cite{hg2013}.)
So if $\epsilon$ is small enough, then
\[\begin{split}&\Big|\int_{\mathbb{T}^m}g(\varphi)\lambda_{\epsilon}-\int_{\mathbb{T}^m}g(\varphi)d\varphi\Big|\\
&\leqslant \frac{1}{\mu(\Gamma_0)}\Big\{\Big|\int_{(I_0,\varphi_0)\in\Lambda_{\rho}^{\epsilon}}\big[\int_{\mathbb{T}^m}g(\varphi)d\lambda_{\epsilon}^{I_0,\varphi_0}-\int_{\mathbb{T}^m}g(\varphi)d\varphi\big]d\mu(I_0,\varphi_0)\Big|
\\
&+\Big|\int_{(I_0,\varphi_0)\in\Gamma_0\setminus\Lambda_{\rho}^{\epsilon}}\big[\int_{\mathbb{T}^m}g(\varphi)d\lambda_{\epsilon}^{I_0,\varphi_0}-\int_{\mathbb{T}^m}g(\varphi)d\varphi]d\mu(I_0,\varphi_0)\Big|\Big\}\leqslant2 \rho.
\end{split}\]

That is ,
\begin{equation}
\Big|\int g(\varphi)\lambda_{\epsilon}-\int g(\varphi) d\varphi\Big|\longrightarrow0\quad\mbox{as}\quad\epsilon\rightarrow 0,
\end{equation}
for any Lipschitz function as above. Hence, the probability measure $\lambda_{\epsilon}$ converges weakly to the Haar measure $d\varphi$ (see \cite{dud2002}). This proves the assertion (ii).

\subsection{Consistency conditions}
Assume the  $\epsilon$-quasi-invariant measure $\mu$ depends on $\epsilon$, i.e. $\mu=\mu_{\epsilon}$.   Using again the Anosov arguments,  we have for the measure~$\mu_{\epsilon}$ that 
 \[\mu_{\epsilon}\big(N(\tilde{T},\epsilon,\delta,T_0)\big)\leqslant \frac{e^{C_{\epsilon}(M_1)}}{\tilde{T}}\mu_{\epsilon}\big(\Gamma_1\setminus E(\delta,T_0)\big).
      \] 
      It is easy to see that assertion (i) of Theorem 0.2 holds, with $\mu$ replace by $\mu_{\epsilon}$,  if the  following consistency conditions are satisfied:
      
      1) For any $\delta>0$, $\mu_{\epsilon}(\Gamma_1\setminus E(\delta,\epsilon^{-\sigma})) $ go to zero with $\epsilon$.
      
      2) $C_{\epsilon}(M_1)$ is uniformly bounded with respect to $\epsilon$
      
      In subsection 3.2, we can see that for assertion (ii) of Theorem 0.2 to hold, one more conditions should be added to the family $\{\mu_{\epsilon}\}_{\epsilon\in(0,1)}$. That is,
      
      3) For any $m$, $n\in\mathbb{N}$,  $\mu_{\epsilon}\big(\Upsilon_{m,n}(\epsilon)\big)$ (see (\ref{b_m})) goes to zero with $\epsilon$.      
      
\section{On existence of $\epsilon$-quasi-invariant measures}
In this section we prove that if Assumption A holds with $\zeta_0\geqslant2$, then there exist $\epsilon$-quasi-invariant measures for the perturbed KdV (\ref{pkdv1}) on the space $H^p$, where $p\geqslant 3$ is an integer.
Through  this section, we suppose that  $\zeta_0=2$, $3\leqslant p\in\mathbb{N}$ and $p^{\prime}=0$.  Our presentation closely follows Chapter 4 of the book \cite{zhidkov}.

Let $\eta_p$ be the centered Gaussian measure on $H^p$ with correlation operator $\partial_x^{-2}$. Since $\partial^{-2}_x$ is an operator of trace type, then  $\eta_p$ is a well-defined probability measure on $H^p$. 

As is known, for solutions of KdV, there are countably many conservation laws $\mathcal{J}_n(u)$, $n\geqslant0$of the form $\mathcal{J}_n=\frac{1}{2}||u||_n^2+J_{n-1}(u)$, where $J_{-1}(u)=0$ and for $n\geqslant1$,
\begin{equation}
J_{n-1}(u)=\int_{\mathbb{T}}\big\{c_nu(\partial_x^{n-1}u)^2+\mathcal{Q}_n(u,\dots,\partial^{n-2}_xu)\big\}dx\,.
\label{law-form}
\end{equation}
 where  $c_n$ are real constants, and $\mathcal{Q}_n$ are polynomial in their arguments (see, p.209 in \cite{kjp2003} for exact form of the conservation laws). By induction we get from these relations that 
\begin{equation}
||u||_n^2\leqslant 2\mathcal{J}_n+C(\mathcal{J}_{n-1},\dots,\mathcal{J}_0),\quad n\geqslant 1,
\label{relation1}
\end{equation}
where $C$ vanishes with $u(\cdot)$.

 From (\ref{law-form}) we know that  the functional $J_p$ is  bounded on bounded sets in $H^p$.
We consider the measure $\mu_p$ defined by
\begin{equation}\mu_p(\Omega)=\int_{\Omega}e^{-J_p(u)}d\eta_p(u),
\label{gibbs-m}
\end{equation}
for every Borel set $\Omega\subset H^p$.
 This measure  is  regular in the sense of Definition~0.1 and non-degenerated in the sense that its support contains the whole space $H^p$ (see, e.g. Chapter 9 in \cite{Bogachev}). Moreover, it is invariant for KdV \cite{zhidkov}.
  
  The main result of this section is the following theorem:
  \begin{theorem}
  The measure $\mu_p$ is $\epsilon$-quasi-invariant for perturbed KdV equation~(\ref{pkdv1}) on the space $H^p$.
  \label{4.1}
  \end{theorem}
  To prove this theorem, we follow a classical procedure based on finite dimensional approximation (see, e.g. \cite{zhidkov}).  
  
  Let us firstly write equation (\ref{pkdv1}) using the slow time $\tau=\epsilon t$,
  \begin{equation}
  \dot{u}=\epsilon^{-1}(-u_{xxx}+6uu_x)+f(u),
  \label{pkdv2}
  \end{equation}
  where $\dot{u}=\frac{du}{d\tau}$.
 By Assumption A, for each $u_0\in B^u_p(M)$, the equation (\ref{pkdv2}) has a unique solution $u(\cdot)\in C([0,T], H^p)$ and $||u(\tau)||_p\leqslant C(||u_0||_p,T)$  for all $\tau\in[0,T]$.
 
  Denote $\mathbb{L}_m$ the subspace of $H^p$, spanned by the basis vectors $\{e_1,e_{-1},\dots,e_m,e_{-m}\}$. Let $\mathbb{P}_m$ be the orthogonal projection of $H^p$ onto $\mathbb{L}_m$ and $\mathbb{P}_m^{\bot}=Id-\mathbb{P}_m$. For any $u\in H^p$,  denote $u^m=\mathbb{P}_mu\in\mathbb{L}_m$. We will identify  $\mathbb{P}_{\infty}$ with $Id$ and $u^{\infty}$ with $u$.
  Consider the problem
  \begin{equation}
  \dot{u}^m=\epsilon^{-1}\big[-u^m_{xxx}+6\mathbb{P}_m(u^mu^m_x)\big]+\mathbb{P}_m(f(u^m)),\quad u^m(x,0)=\mathbb{P}_mu_0(x).
  \label{finite-s}
  \end{equation}
  Clearly, for each $u_0\in H^p$ this problem has a unique solution $u^m(\cdot)\in C([0,T^{\prime}], \mathbb{L}_m)$ for some $T^{\prime}>0$.
   \begin{proposition}
  Let $u_0\in H^p$ and $u_0^m\in\mathbb{L}_m$ such that $u_0^m$ strongly converge to~$u_0$  in $H^p$ as $m\to+\infty$.
  Then as $m\to+\infty$,
  $$u^m(\cdot)\to u(\cdot)\quad\text{in}\quad C([0,T],H^p),$$
  where $u(\cdot)$ is  the solution of equation (\ref{pkdv1}) with initial datum $u(0)=u_0$ and $u^m(\cdot)$ is the solution of problem (\ref{finite-s}) with initial condition $u^m(0)=u^m_0$.
  \label{4.2}
  \end{proposition}
  In this result, as well as in the Lemmas \ref{4.5}-\ref{4.7} below, the rate of convergence depends on the small parameter $\epsilon$.
  
To prove this proposition, we start with  several lemmas. For any  $n,m\in\mathbb{N}$, we have for the solution $u^m(\tau)$ of problem(\ref{finite-s})
 \[\begin{split}
 \frac{d}{d\tau}\mathcal{J}_n(u^m(\tau))&=\big\langle \nabla_u \mathcal{J}_n(u^m),\dot{u}^m(\tau)\big\rangle\\&=\big\langle\nabla_u \mathcal{J}_n(u^m),\epsilon^{-1}[-u^m_{xxx}+\mathbb{P}_m(u^mu^m_x)]+\mathbb{P}_m[f(u^m)]\big\rangle
 \end{split}
\]
Here $\nabla_u$ stands for the $L_2$-gradient with respect to $u$. Since $\mathcal{J}_n$ is a conservation law of KdV, then  $\langle\nabla_u\mathcal{J}_n(u^m), -u^m_{xxx}+u^mu^m_x\rangle=0$. So
\begin{equation}\frac{d}{d\tau}\mathcal{J}_n(u^m)=-\epsilon^{-1}\big\langle\nabla_u\mathcal{J}_n(u^m),\mathbb{P}^{\bot}_m(u^mu^m_x)\big\rangle+\big\langle \nabla_u \mathcal{J}_n(u^m),\mathbb{P}_m[f(u^m)]\big\rangle.
\label{d-t1}
\end{equation}
We denote the first term in the right hand side by $\epsilon^{-1}\mathcal{E}_n(u^m)$ and the second term by $\mathcal{E}_n^f(u^m)$.

\begin{lemma}
There exist continuous functions $\gamma_n(R,s)$ and $\gamma^{\prime}_n(R,s)$ on \mbox{$\mathbb{R}^2_+=\{(R,s)\}$ } such that they are non-decreasing in the second variable $s$,  vanish if $s=0$, 
 and 
 \begin{eqnarray}
  &&|\mathcal{E}_{n}^f(u^m)|\leqslant\gamma^{\prime}_n(||u^m||_{n-1},||u^m||_{n-1})\label{ef-b1},\\
 &&\begin{split}|\mathcal{E}_n(u^m)|\leqslant&\gamma_n\Big(||u^m||_{n-1},\\
 &\max_{\tiny \begin{array}{l}0\leqslant i,j\leqslant n-1,\\i+j\neq 2n-2\end{array}}||\mathbb{P}^{\bot}_m[\partial^i_xu^m\partial^j_xu^m]||_0+||\mathbb{P}_m^{\bot}(u^mu_x^m)||_1\Big).
 \end{split}
 \label{ef-bound2}
 \end{eqnarray}for all $n=3,4,\dots.$ For $n=2$ equality (\ref{ef-b1}) still holds, and 
 \begin{equation}
 |\mathcal{E}_2(u^m)|\leqslant C_2(||u^m||_1)||u^m||_2^2+C_2^{\prime}(||u^m||_1).
 \end{equation}
 \label{4.3}
 \end{lemma}
 \begin{proof}
Since $f(u)$ is 2-smoothing, from (\ref{law-form}) and (\ref{d-t1}) we know that 
\[|\mathcal{E}_n^f(u^m)|\leqslant \gamma_n^{\prime}(||u^m||_{n-1},||u^m||_{n-1}),\]
where $\gamma_n^{\prime}(\cdot,\cdot)$ is a continuous function satisfying the requirement  in the statement of the lemma.

For the quantity $\mathcal{E}_n(u^m)$, by (\ref{law-form}) and (\ref{d-t1})  we have 
\[\begin{split}\mathcal{E}_n(u^m)&=\int_{\mathbb{T}}\Big\{6(-1)^n(\partial^{2n}_xu^m)\mathbb{P}^{\bot}_m(u^mu^m_x)+6c_n\mathbb{P}_m^{\bot}(u^mu^m_x)(\partial^{n-1}_xu^m)^2\\
&\quad\quad+(-1)^{n-1}12c_n\partial^{n-1}_x(u^m\partial^{n-1}_xu^m)\mathbb{P}^{\bot}_m(u^mu^m_x)\\
&\quad\quad+6\sum_{i=0}^{n-2}\frac{\partial \mathcal{Q}_n(u^m,\dots,\partial^{n-2}_xu^m)}{\partial (\partial^i_xu^m)}\partial^i_x\mathbb{P}^{\bot}_m(u^mu^m_x)\Big\}dx\\
&=0+\int_{\mathbb{T}}\Big\{6c_n\mathbb{P}^{\bot}_m(u^mu^m_x)(\partial_x^{n-1}u^m)^2\\
&\quad\quad+12c_n\mathbb{P}^{\bot}_m(u^m\partial^{n-1}_xu^m)[\partial_x\sum_{i=0}^{n-3}C_{n-2}^i\partial^{n-2-i}_xu^m\partial_x^{i+1}u^m]\\
&\quad\quad+6\sum_{i=0}^{n-2}\frac{\partial \mathcal{Q}_n(u^m,\dots,\partial^{n-2}_xu^m)}{\partial (\partial^i_xu^m)}\partial_x^i\mathbb{P}^{\bot}_m(u^mu^m_x)\Big\}dx.\end{split}
\]
 Hence we prove the assertion of the lemma.
\end{proof}
 
 \begin{lemma}For every $u_0\in H^p$, there exist $\tau_1(||u_0||_0)>0$ and a continuous non-decreasing $\epsilon$-depending function $\beta^{\epsilon}_p(s)$ on $[0,+\infty)$ such that   the value $||u^m(\tau)||_p$ are  bounded by the quantity $\beta^{\epsilon}_n(||u_0||_p)$, uniformly in  $m=1,2,\dots$ and  $\tau\in[0,\tau_1]$.
 \label{4.4}
 \end{lemma}
 \begin{proof} 
Let $M=\max\{||u_0||_0,1\}$.  It is easy to verify that
\[\frac{d}{d\tau}||u^m||_0^2=2\langle u^m, \mathbb{P}_m(f(u^m))\rangle\leqslant2||u^m||_0^2+C(2M),\]
if $||u^m||_0\leqslant 2M$. Therefore for a suitable $\tau_1=\tau_1(||u_0||_0)>0$ and  all $\tau\in[0,\tau_1]$, we have $||u^m(\tau)||_0\leqslant 2 M$.
For the quantity $\mathcal{J}_1(u^m)$ and $\tau\in[0,\tau_1]$,
\[\frac{d}{d\tau}\mathcal{J}_1(u^m)=\langle\nabla_u \mathcal{J}_1(u^m),\mathbb{P}_m f(u^m)\rangle\leqslant C_1(2M).
\]
Therefore $\mathcal{J}_1(u^m(\tau))\leqslant C_1\tau+\mathcal{J}_1(u^m(0))$.  So $||u^m(\tau)||_1\leqslant \beta_1(||u_0||_1)$. Similarly,  by  Lemma \ref{4.3} and inequality (\ref{relation1}), we have for $\tau\in[0,\tau_1]$,
\[\frac{d}{d\tau}\mathcal{J}_2(u^m(\tau))\leqslant \epsilon^{-1}C_2[\beta_1(||u_0||_1)]\mathcal{J}(u^m(\tau))+C_2^{\prime\prime}[\epsilon^{-1},\beta_1(||u_0||_1)].\]
By Gronwall's lemma and relation (\ref{relation1}), we obtain $||u^m(\tau)||_2\leqslant \beta^{\epsilon}_2(||u_0||_2)$.
In the view of Lemma~\ref{4.3}, 
we have $$\mathcal{J}_n(u^m(\tau))\leqslant \mathcal{J}_n(u^m(0))+\tau C_n[\epsilon^{-1}, \beta_{n-1}^{\epsilon}(||u_0||_{n-1})],$$ for $n=3,\dots,p$. 
Hence $\max_{\tau\in[0,\tau_1]}||u^m(\tau)||_p\leqslant \beta^{\epsilon}_p(||u_0||_p)$.
 \end{proof}
 
 Below, we will denote by $\tau_1$ the quantity $\min\{\tau_1(||u_0||_0), T\}$. 
 \begin{lemma}As $ m\to\infty,$ 
 \[||u^m(\tau)-u(\tau)||_{p-1}\to0, \]
 uniformly in $\tau\in[0,\tau_1]$.
 \label{4.5}
 \end{lemma}
 \begin{proof} Denote $w=u^m-u$. Using that $\langle \partial_x^ju^m,\mathbb{P}_m^{\bot}u^{\prime}\rangle=0$ for any $j$ and each $u^{\prime}\in H^0$, we get:
 \[\begin{split}
 &\frac{1}{2}\frac{d}{d\tau}||w||_{p-1}^2\\&=\Big\langle \partial^{p-1}_xw, \;\;\partial^{p-1}_x\big[\epsilon^{-1}\big(-w_{xxx}+6\mathbb{P}_m(u^mu^m_x)-6uu_x\big)+\mathbb{P}_m(f(u^m))-f(u)\big]\Big\rangle\\
 &=3\epsilon^{-1}\Big\langle \partial^{p-1}_xw,\;\;\partial^p_x\big[(u^m)^2-u^2\big]\Big\rangle+3\epsilon^{-1}\Big\langle\mathbb{P}_m^{\bot}(\partial^{p-1}_xu),\partial^p_x[(u^m)^2]\Big\rangle\\
& \quad+\Big\langle \partial^{p-1}_xw,\;\;\partial^{p-1}_x\big[\mathbb{P}_m(f(u^m))-f(u)\big]\Big\rangle.
 \end{split}
 \]
 Using Sobolev embedding and integration by part, we have 
 \[\begin{split}\Big\langle \partial^{p-1}_xw,\partial^p_x\big[(u^m)^2-u^2\big]\Big\rangle&=\sum_{i=0}^pC_p^i\int_{\mathbb{T}}\partial^{p-1}_xw\partial^{p-i}_xw\partial^{i}_x(u^m+u)dx\\
 &\leqslant-\int_{\mathbb{T}}\partial_x(u^m+u)(\partial^{p-1}_xw)^2dx+\sum_{i=1}^pC_p^i||w||_{p-1}^2||u^m+u||_p\\
 &\leqslant C(||u||_p,||u^m||_p)||w||^2_{p-1}
 \end{split}
 \]
 Therefore,
 \[\frac{d}{d\tau}||w||_{p-1}^2\leqslant C_1(\epsilon^{-1},||u^m||_n)||\mathbb{P}_m^{\bot}u||_p+||\mathbb{P}_m^{\bot}f(u)||_p+C_2(\epsilon^{-1},||u||_n,||u^m||_n)||w||_{p-1}^2.
 \]
 Since $||\mathbb{P}_m^{\bot}(u)||_p$ and $||\mathbb{P}_m^{\bot}(f(u))||_p$ go to zero as $m\to\infty$ for each $\tau\in[0,\tau_1]$ and as they are uniformly bounded on $[0,\tau_1]$ by Lemma \ref{4.4}, we have for $\tau\in[0,\tau_1]$,
 \[||w||_{p-1}^2(\tau)=||w(0)||_{p-1}^2+\int_0^{\tau}C(\epsilon^{-1},||u_0||_p)||w||^2ds+a_m(\epsilon^{-1},\tau),\]
 where $a_m(\epsilon^{-1},\tau)\to0$ as $m\to\infty$. So the assertion of the lemma follows form \mbox{Gronwall's lemma}. 
 \end{proof}
 
 \begin{lemma}
 Let $\tau^m\in[0,\tau_1]$ such that $\tau^m\to \tau^0\in[0,\tau_1]$, then 
 \[||u^m(\tau^m)-u(\tau_0)||_p\to0\quad\text{ as} \quad m\to\infty.\]
 \label{4.6}
 \end{lemma}
 \begin{proof}
 We firstly prove $\mathcal{J}_p(u^m(\tau^m))\to \mathcal{J}_p(u(\tau_0)$ as $m\to\infty$. Indeed, for $m\leqslant +\infty$, by (\ref{d-t1}), we have 
 \[\mathcal{J}_p(u^m(\tau^m))=\mathcal{J}_p(u^m(0))+\int_0^{\tau^m}[\epsilon^{-1}\mathcal{E}_p(u^m(s))+\mathcal{E}^f_p(u^m(s))]ds\]
 Since $f(u)$ is 2-smoothing, the second term in the integrand is continuous in $H^{p-1}$. So, in the view of Lemma \ref{4.5}, we  only need to prove that the first term goes to zero as $m\to\infty$.
 Due to Lemma \ref{4.3}, we only need to show that uniformly in $\tau\in[0,\tau_1]$,
\[ \lim_{m\to\infty}||\mathbb{P}_m^{\bot}(\partial_x^i u^m(\tau)\partial^j_xu^m(\tau)||_0+\lim_{m\to\infty}||\mathbb{P}_m^{\bot}u^m(\tau)u^m(\tau)_x||_1=0,\]
where $0\leqslant i,j\leqslant p-1$ and $i+j\neq 2p-2$. For the first term in the left hand side, we have 
\begin{equation}\begin{split}&||\mathbb{P}_m^{\bot}(\partial^i_xu^m(\tau)\partial_x^ju^m(\tau))||_0\\
&\leqslant||\mathbb{P}_m^{\bot}(\partial^i_xu^m(\tau)\partial^j_xu^m(\tau)-\partial^i_xu(\tau)\partial^j_xu(\tau))||_0+||\mathbb{P}_m^{\bot}(\partial^i_xu(\tau)\partial^j_xu(\tau))||_0.
\end{split}
\label{p-u-b1}
\end{equation}
By Lemma \ref{4.5}, the first term in the r.h.s of (\ref{p-u-b1}) goes to zero as $m\to\infty$, uniformly in $\tau\in[0,\tau_1]$. Since $u(\cdot)\in C([0,\tau_1], H^p)$, then $\partial^i_xu(\cdot)\partial^j_xu(\cdot)\in C([0,\tau_1], H^0)$. Therefore, the quantity $||\mathbb{P}_m^{\bot}(\partial^i_xu(\tau)\partial^j_xu(\tau)||_0\to0$ as $m\to\infty$, uniformly in $\tau\in[0,\tau_1]$.

In the same way $\lim_{m\to\infty}||\mathbb{P}_m^{\bot}u^mu^m_x||_1=0$.
Therefore, we have 
\[\lim_{m\to\infty}\mathcal{J}_p(u^m(\tau^m))=\mathcal{J}_p(u(0))+\int_0^{\tau^0}\langle\nabla_u\mathcal{J}_p(u(s)),f(u(s))\rangle ds=\mathcal{J}_p(u(\tau^0)).\]
Since the quantity $\mathcal{J}_p(u)-||u||_p^2/2$ is continuous in $H^{p-1}$, we have $$\lim_{m\to\infty}||u^{m}(\tau^m)||_p=||u(\tau_0)||_p.$$ 
 The assertion of the lemma follows from the fact that weak convergence plus norm convergence imply strong convergence.
 \end{proof}
 \begin{lemma}
 As $m\to\infty$, $||u^m(\tau)-u(\tau)||_p\to0$ uniformly for  $\tau\in[0,\tau_1]$.
 \label{4.7}
 \end{lemma}
 \begin{proof} Assume the opposite. Then there exists $\delta>0$ such that for each $m\in\mathbb{N}$, there exists $\tau^m\in[0,\tau_1]$ satisfying
\begin{equation} ||u^m(\tau^m)-u(\tau^m)||_p\geqslant\delta.
\label{more-delta}
\end{equation}
Take a  subsequence $\{m_k\}$ such that $\tau^{m_k}\to\tau^0\in[0,\tau_1]$ as $m_k\to\infty$. By Lemma~4.6, we have 
\[\begin{split}&\quad\lim_{m_k\to\infty}||u^{m_k}(\tau^{m_k})-u(\tau^{m_k})||_p\\&=\lim_{m_k\to\infty}(||u^{m_k}(\tau^{m_k})-u(\tau^0)||_p+||u(\tau^{m_k})-u(\tau^0)||_p)=0.
\end{split}\]
This contradicts the inequality (\ref{more-delta}). So the assertion of the Lemma holds.
 \end{proof}
 If $T=\tau_1$, Proposition \ref{4.2} is proved. Otherwise,  we j iterate the above procedure by letting the  initial datum to be $u(\tau_1)$. This completes the proof of Proposition~\ref{4.2}.
 
Apart from Proposition \ref{4.2},  we will need two more results to prove Theorem \ref{4.1}.
 
\begin{proposition}
For each  $u_0\in H^p$ and any $\nu>0$, there exists $\delta>0$ such that 
\[||u^m(\tau)-u^m_1(\tau)||_p<\nu,\]
uniformly in $m=1,2,\dots$, $\tau\in[0,T]$ and for every solution $u_1^m(\cdot)$ of problem (\ref{finite-s}) with initial data $u_1^m(0)$, satisfying
\[||u^m(0)-u^m_1(0)||_p<\delta,\]
 (here $u^m(\cdot)$ is the solution of (\ref{finite-s}) with initial data $\mathbb{P}_mu_0$).
 \label{4.8}
 \end{proposition}
 \begin{proof}
 Assume the contrary. Then there exists $\nu>0$ such that for each $\delta>0$, there exists  $m\in\mathbb{N}$, $u_1\in \mathbb{L}_m$ and $\tau^m\in[0,T]$ satisfying
 \begin{equation}||u^m_1(\tau^m)-u^m(\tau^m)||_p\geqslant\nu\quad \text{and}\quad ||u^m_1(0)-u^m(0)||_p<\delta.
 \label{p-delta-more}
 \end{equation}
 Hence there exists a subsequence $\{m_k\}$ such that $||u^{m_k}_1-\mathbb{P}_{m_k}u_0||_p\to0$ as $m_k\to\infty$.  Therefore $\lim_{m_k\to\infty}||u_1^{m_k}-u_0||_p=0$. 
 By Proposition \ref{4.2}, we known that 
 $$|u^{m_k}_1(\tau^{m_k})-u^{m_k}(\tau^{m_k})||_p\leqslant ||u^{m_k}_1(\tau^{m_k})-u(\tau^{m_k})||_p+||u^{m_k}(\tau^{m_k})-u(\tau^{m_k})||_p\to 0,$$ 
 as $m_k\to\infty$. This contradicts  the first inequality of (\ref{p-delta-more}). Proposition \ref{4.8} is proved.
 \end{proof}
  
 \begin{lemma}
Let $u_0\in H^p$. Then for any $\delta>0$, there exist $r>0$ and $m_0>0$ such that for each $m\geqslant m_0$ and $\bar{u}(0)\in \dot{B}_p^u(u_0,r)$,  the quantity 
 \[\epsilon^{-1}|\mathcal{E}_{p+1}(\bar{u}^m(\tau))|\leqslant \delta,\]
 for all  $\tau\in[0,T]$.
 \end{lemma}
 \begin{proof}
 In the view of Lemma \ref{4.3}, we only need to show for each $\delta_{\epsilon}>0$, there exist $r>0$ and $m_0>0$ such that for every $\bar{u}_0\in \dot{B}_p^u(u_0,r)$, and $m\geqslant m_0$, we have for $\tau\in[0,T],$
 \begin{equation}\max_{0\leqslant i,j\leqslant p,i+j\neq 2p}||\mathbb{P}^{\bot}_m[\partial^i_x\bar{u}^m(\tau)\partial^j_x\bar{u}^m(\tau)]||_0+||\mathbb{P}_m^{\bot}(\bar{u}^m(\tau)\bar{u}_x^m(\tau))||_1<\delta_{\epsilon}.
 \label{mix-delta}
 \end{equation}
 Here $\bar{u}^m(\tau)$ is the solution of problem (\ref{finite-s}) with initial datum $\bar{u}^m(0)=\mathbb{P}_m\bar{u}_0$.
 
 For the first term, we have
 \[\begin{split}&||\mathbb{P}^{\bot}_m[\partial^i_x\bar{u}^m(\tau)\partial^j_x\bar{u}^m(\tau)]||_0\\
&\leqslant ||\partial^i_xu^m(\tau)\partial^j_xu^m(\tau)-\partial^i_x\bar{u}^m(\tau)\partial^j_x\bar{u}^m(\tau)||_0\\
 &\quad+||\partial^i_xu^m(\tau)\partial^j_xu^m(\tau)-\partial^i_xu(\tau)\partial^j_xu(\tau)||_0+||\mathbb{P}_m^{\bot}[\partial^i_xu(\tau)\partial^j_xu(\tau)]||_0.
 \end{split}
 \]
 By Proposition \ref{4.2} and the fact that $\partial^i_xu(\cdot)\partial^j_xu(\cdot)\in C([0,T],H^0)$, the second and the third terms on the right hand side of this inequality converge to zero as $m\to\infty$, uniformly in $\tau\in[0,T]$. From Proposition \ref{4.8}, we know that there exists $r>0$ such that the first term is smaller than $\delta_{\epsilon}/2$ for all $\bar{u}\in \dot{B}_r^p(u_0)$  and uniformly in $m\in\mathbb{N} $ and $\tau\in[0,T]$. Estimating in this way the term $||\mathbb{P}_m^{\bot}(\bar{u}^m\bar{u}^m_x)||_1$,  we obtain inequality (\ref{mix-delta}).
 Hence we prove the assertion of the lemma.
 \end{proof}
  
 We now begin to prove Theorem 4.1.

 Consider the following Gaussian measure $\eta_p^m$ on the subspaces $\mathbb{L}_m\subset H^p$:
 \[\begin{split}d\eta_p^m:&=\prod_{i=1}^{m}(2\pi)^{2p}i^{2p+1}\exp{-\frac{(2\pi i)^{2p+2}(\hat{u}_i^2+\hat{u}_{-i}^2)}{2}}d\hat{u}_id\hat{u}_{-i}\\
 &=c(m) \exp{\frac{-||u^m||_{p+1}^2}{2}}d\hat{u}_1d\hat{u}_{-1}\dots d\hat{u}_md\hat{u}_{-m},
 \end{split}\]
 where $u^m:=\sum_{i=1}^m(\hat{u}_ie_i+\hat{u}_{-i}e_{-i})\in\mathbb{L}_m$ and  $d\hat{u}_{\pm i}$, $i\in\mathbb{N}$, is the Lebesgue measure on $\mathbb{R}$.
 Obviously, $\eta^m_p$ is a Borel measure on $\mathbb{L}_m$. Then we have obtained a sequence of Borel measure $\{\eta^m_p\}$ on $H^p$ (see, e.g. \cite{zhidkov}).
  We set $$\mu_p^m(\Omega)=\int_{\Omega}e^{-J_p(u)}d\eta^m_p,$$ for every Borel set $\Omega\in H^p$. Then $\mu_p^m$ are well defined Borel measure on $H^p$. Clearly $$d\mu_p^m=c(m)e^{-\mathcal{J}_{p+1}(u^m)}d\hat{u}_1d\hat{u}_{-1}\dots d\hat{u}_{m}d\hat{u}_{-m}.$$
  \begin{lemma}(\cite{zhidkov})
  The sequence of Borel measures $\mu_p^m$ in $H^p$ converges weakly to the measure $\mu_p$ as $m\to\infty$.
  \end{lemma}

 Rewrite the system (\ref{finite-s}) in the variables 
 $$\hat{u}^m=(\hat{u}_1,\hat{u}_{-1},\dots,\hat{u}_{m},\hat{u}_{-m}),$$
 where $u^m=\sum_{j}^m(\hat{u}_je_j+\hat{u}_{-j}e_{-j})$: \begin{equation}
 \frac{d}{d\tau}\hat{u}_j=-2\pi j\epsilon^{-1}\frac{\partial \mathcal{J}_1(\hat{u}^m)}{\partial\hat{u}_{-j}}+f_j(\hat{u}^m),\quad j=\pm1,\dots,\pm m.
 \label{finite-vector}
 \end{equation}
 where $\mathbb{P}_m f(\hat{u}^m)=\sum_{j=1}^m(f_j(\hat{u}^m)e_j+f_{-j}(\hat{u}^m)e_{-j})$.
 Let $S_m^{\tau}$, $\tau\in[0,T]$,  be the flow map of equation (\ref{finite-vector}).  For any Borel set $\Omega\subset H^p$, let $S_m^{\tau}(\Omega)=S_m^{\tau}(\mathbb{P}_m(\Omega))$.  By the Liouville Theorem and (\ref{d-t1}), we have
 \begin{equation}
 \frac{d}{d\tau}\mu^m(S_m^{\tau}(\Omega))=\int_{S_m^{\tau}(\Omega)}\Big[\epsilon^{-1}\mathcal{E}_{p+1}(u^m)+\mathcal{E}_{p+1}^f(u^m)+\sum_{i=-m,i\neq0}^m\frac{\partial f_i}{\partial\hat{u}_i})\Big]d\mu^m.
  \label{quasi-1}
 \end{equation}
 
 Denote $S^{\tau}_{\epsilon}$, $\tau\in [0,T],$ to be the flow map of equation (\ref{pkdv2}) on the space $H^p$. Fix any $M>0$. By Assumption A, there exists $M_1$ such that 
 \[S^{\tau}_{\epsilon}(B_p^u(M))\subset B_p^u(M_1).\]
Since $f(u)$ is 2-smoothing,  then  on the ball $B_p(2M_1)$ we have $|f_i(u)|\leqslant |i|^{-p-2}C(2M_1)$.   By Cauchy inequality, we have $|\partial f_i/\partial\hat{u}_i|\leqslant C(2M_1)i^{-2}$  on the ball $B_p(M_1)$. So we have
\begin{equation}|\mathcal{E}_{p+1}^f(u^m)+\sum_{i=-m,i\neq0}^m\frac{\partial f_i}{\partial\hat{u}_i}(u^m)|\leqslant C(M_1), \;\;\forall m\in\mathbb{N}\; \;\text{ and}\;\;\forall u^m\in B_p^u(M_1).
\label{quasi-2}
\end{equation}

  Now fix $\tau_0\in[0,T]$.
 Take  an open set $\Omega\subset B_p^u(M)$.  For any $\delta>0$,  there exists a compact set $K\subset \Omega$ such that $\mu_p(\Omega\setminus K)<\delta$. Let $K_{1}=S_{\epsilon}^{\tau_0}(K)$. Then the set 
 $K_{1}$  also is compact and $K_{1}\subset S_{\epsilon}^{\tau_0}(\Omega)=\Omega_{1}$. 
 Define
 \[\alpha=\min\{dist(K,\partial \Omega); dist(K_{1},\partial \Omega_{1})\},\]
 where $dist(A,B)=\inf_{u\in A,v\in B}||u-v||_p$ and $\partial A$ is the boundary of the set $A\subset H^p$.
 
 Clearly $\alpha>0$. 
 By Proposition 4.8 and Lemma 4.9, for each $u_0\in K$, there exists a $m_{u_0}>0$ and an open ball $\dot{B}_p^u(u_0,r_{u_0})$ of radius $r_{u_0}>0$   such that 
 \begin{equation}||u^m(s)-\bar{u}^m(s)||_p\leqslant \alpha/3 \quad\text{and}\quad |\epsilon^{-1}\mathcal{E}_{p+1}(\bar{u}^m)|\leqslant C(M_1)/2,
 \label{quasi-3}
 \end{equation}
 for all $\bar{u}\in B_p^u(u_0,r_{u_0})$, $m\geqslant m_0$ and $s\in[0,\tau_0]$.  Let $B_1,\dots,B_l$ be the finite covering of the compact set $K$ by such balls. Let $$D=\cup_{i=1}^l B_i \quad\text{and}\quad \Omega_{\alpha/3}:=\{u\in \Omega_1|\;\; dist(u,\partial \Omega_1)\geqslant \alpha/3\}.$$
  By Proposition 4.2, $$S^{\tau_0}_m(D)\subset\Omega_{\alpha/3}, $$
 for all large enough $m\in\mathbb{N}$.  From inequalities (\ref{quasi-1}), (\ref{quasi-2}) and (\ref{quasi-3}), we know that if $m$ is sufficiently large, then 
 \[ e^{-3C(M_1)\tau_0/2}\mu^m_p(D)\leqslant \mu_p^m(S^{\tau_0}(D))\leqslant e^{3C(M_1)\tau_0/2}\mu_p^m(D).
 \]
 By Lemma 4.10, we have
 \[\begin{split}\mu_p(\Omega)&\leqslant \mu_p(D)+\delta\leqslant \liminf_{m\to\infty}\mu_p^m(D)+\delta\\
 &\leqslant\liminf_{m\to\infty}e^{3C(M_1)\tau_0/2}\mu_p^m(S^{\tau_0}_m(D))+\delta\leqslant \limsup_{m\to\infty}e^{3C(M_1)\tau_0/2}\mu_p^m(\Omega_{\alpha/3})+\delta\\
 &\leqslant e^{3C(M_1)\tau_0/2}\mu_p(\Omega_1)+\delta.
 \end{split}\]
 Here we have used the Portemanteau theorem.
 Since $\delta$ was chosen arbitrarily, it follows that 
 \[\mu_p(\Omega)\leqslant e^{3C(M_1)\tau_0/2}\mu_p(S^{\tau_0}_{\epsilon}(\Omega)).\]
 Similarly, $\mu_p(S^{\tau_0}(\Omega))\leqslant e^{3C(M_1)\tau_0/2}\mu_p(\Omega)$.  As $\tau_0\in[0,T]$ is fixed arbitrarily, Theorem 4.1 is proved.
 
 \begin{remark}
 The measure $\mu_p$ is also $\epsilon$-quasi-invariant for the following perturbed KdV equations on $H^p$:
 \begin{eqnarray}
 &&\dot{u}+\epsilon^{-1}(u_{xxx}-6uu_x)=\partial_xu,\label{pkdv3}\\
 &&\dot{u}+\epsilon^{-1}(u_{xxx}-6uu_x)=\partial^{-1}_xu.\label{pkdv4}
 \end{eqnarray}
 \label{4.11}
  \end{remark}
 Indeed, consider the following finite dimensional system  corresponding to equation (\ref{pkdv3})  as in problem  (\ref{finite-s}):
 \begin{equation}
 \dot{u}^m=\epsilon^{-1}\big[-u^m_{xxx}+6\mathbb{P}_m(u^mu^m_x)]+\partial_xu^m, \quad u^m(0)=\mathbb{P}_mu_0.
 \label{finite-pkdv3}
 \end{equation}
 Let us investigate the quantity $\frac{d}{d\tau}\mathcal{J}_n(u^m)$, $n\geqslant3$, for equation (\ref{finite-pkdv3}):
 \[\frac{d}{d\tau}\mathcal{J}_n(u^m)=\epsilon^{-1}\mathcal{E}_n(u^m)+\langle\nabla_u\mathcal{J}_n(u^m),\partial_xu^m\rangle.\]
 For the first term, see in Lemma \ref{4.3}. For the second term,
 \[\begin{split}D_n:=\langle\nabla_u\mathcal{J}_n(u^m), \partial_xu^m\rangle=\int_{\mathbb{T}}&\Big\{\partial_x^n u^m\partial_x^{n+1}u^m+c_n\partial_xu^m(\partial^{n-1}_x u^m)^2\\
 &+2c_nu^m\partial^{n-1}u^m\partial^{n}u^m\\
 &+\sum_{i=0}^{n-2}\frac{\partial\mathcal{Q}_n(u^m,\dots,\partial^{n-2}_xu^m)}{\partial(\partial^i_xu^m)}\partial^{i+1}u^m\Big\}dx.
 \end{split}
 \]
 Notice that the first term in right hand side vanishes. For the second and the third terms, 
\[ \int_{\mathbb{T}}c_n[\partial_xu^m(\partial^{n-1}_xu^m)^2+2u^m\partial^{n-1}_xu^m\partial^n_xu^m]dx=c_n\int_{\mathbb{T}}d[u^m(\partial^{n-1}_xu^m)^2]=0. \]
  So we have
 \begin{equation}|D_n|\leqslant C(||u^m||_{n-1}).
 \label{bound-dh}
 \end{equation}
 Note that equation (\ref{finite-pkdv3}) can be written as a Hamiltonian system in coordinates $\hat{u}^m=(\hat{u}_{1},\hat{u}_{-1},\dots, \hat{u}_m,\hat{u}_{-m})$:
 \begin{equation}\frac{d}{d\tau}\hat{u}_j=-2\pi j\epsilon^{-1}\frac{\partial H_1(\hat{u}^m)}{\partial\hat{u}_{-j}},\quad j=\pm 1,\dots,\pm m,\
 \label{finite-h}
 \end{equation}
 where the Hamiltonian $H_1(u)=\mathcal{J}_1(u)-\frac{\epsilon}{2}\int_{\mathbb{T}}u^2dx$.
 Therefore the divergence for the vector field of equation (\ref{finite-h}) is zero. This property and inequality (\ref{bound-dh}) also hold for equation (\ref{pkdv4}). Hence the same proof as in this section  applies to equation (\ref{pkdv3}) and (\ref{pkdv4}), which justifies the claim in the Remark \ref{4.11}.

 \bibliography{pkdv2.bib}

\end{document}